\documentclass[11pt]{amsart}

\usepackage[latin1]{inputenc}
\usepackage{amssymb, amsmath, amsthm,mathrsfs}

\usepackage{amscd}
\usepackage{amssymb}
\usepackage{amsthm}
\usepackage{amsfonts}
\usepackage{amsmath}

\usepackage{verbatim} 

\usepackage{float}

\usepackage{hyperref}
\hypersetup{
    colorlinks,
    citecolor=black,%
    filecolor=black,%
    linkcolor=black,%
    urlcolor=black,
    pdfstartview={FitH},
    pdftitle={Maximum principles for martingale random fields via non-anticipating stochastic derivatives},
    pdfauthor={..},
}

\theoremstyle{plain} \newtheorem{definisjon}{Definition}[section]
\theoremstyle{plain} \newtheorem{teorem}[definisjon]{Theorem}
\theoremstyle{plain} 
\theoremstyle{plain} \newtheorem{korollar}[definisjon]{Corollary}
\theoremstyle{plain} \newtheorem{proposisjon}[definisjon]{Proposition}
\theoremstyle{definition} \newtheorem{remark}[definisjon]{Remark}
\theoremstyle{definition} \newtheorem{antagelse}[definisjon]{Assumption}

\numberwithin{equation}{section}

\newcommand{\ie}{i.e.~}

\newcommand{\E}{\mathbb{E}}
\newcommand{\Prob}{\mathbb{P}}
\newcommand{\ind}{\mathbf{1}}

\newcommand{\RR}{{\mathbb{R}}}
\newcommand{\Zz}{{\mathcal{Z}}}

\newcommand{\Ff}{\mathcal{F}}
\newcommand{\Gg}{\mathcal{G}}
\newcommand{\Dd}{\mathscr{D}}
\newcommand{\FF}{\mathbb{F}}
\newcommand{\GG}{\mathbb{G}}
\newcommand{\Pp}{\mathcal{P}}
\newcommand{\Bb}{\mathcal{B}}
\newcommand{\Ham}{\mathcal{H}}

\newcommand{\II}{\mathcal{I}}
\newcommand{\UU}{\mathcal{U}}
\newcommand{\Md}{D}

\newcommand{\Ht}{\tilde{H}}

\newcommand{\ins}{\,}
\newcommand{\minus}{}
\DeclareMathOperator*{\esssup}{ess\,sup}
\DeclareMathOperator{\Af}{\mathcal{A}^\Ff}

\renewcommand{\theenumi}{\roman{enumi})}

\begin{document}
\title[]{Maximum principles for non-Markovian semi-martingales with jumps and more}
\date{\today }
\author[Sjursen]{Steffen Sjursen}
\address{Steffen Sjursen: Department of mathematics, University of Oslo, PO Box 1053 Blindern, N-0316 Oslo, Norway}

\email[]{steffen.sjursen@cma.uio.no}


\begin{abstract}
We find a maximum principle for general non-Markovian semi-martingales. 
We do so by describing the adjoint processes with non-anticipating stochastic derivatives in a martingale random field setting. 
In the case of the L{\'e}vy processes this extends maximum principles with Malliavin derivatives, in the sense that we replace Malliavin differentiability conditions with weaker and simpler $L_2$-conditions.

As an application we use the maximum principle to solve a portfolio optimization problem for assets with credit risk modeled by doubly stochastic Poisson processes.
\end{abstract}

\maketitle

\section{Introduction}

There are two main approaches to optimization problems, dynamic programming with HJB-type equations or using BSDEs (backward stochastic differential equations). However, for dynamic programming the state equation must be Markovian, while any BSDE-approach requires the existence of the actual BSDE. Here we find a stochastic maximum principle that avoids both of these requirements. 

We consider the performance functional 
\begin{equation}
J(u) = \E \Big[ \int\limits_0^T f_t(u_t,X_t) \ins dt + g(X_T) \Big]
\label{eq:performance_functional_maximum_intro}
\end{equation}
and the associated optimal stochastic control problem, where $u$ is the control and the state process is given by the semi-martingale $X$,
\begin{equation}
X_t^{(u)} = X_0 + \int\limits_0^t b_s(u_s,X_s) \ins ds + \int\limits_0^t \int\limits_\Zz \phi_s(z,u_s,X_{s\minus}) \ins \mu(ds,dz), \quad t\in [0,T],
\label{eq:definitionX}
\end{equation}
where the last integral is with respect to the martingale random field, \cite{Cairoli1975,DiNunno2010}, $\mu$ on $[0,T]\times \Zz$. The choice of \emph{martingale random fields} is made to fit the most general description of the \emph{non-anticipating stochastic derivative} made in \cite{DiNunno2010}. But we must emphasize that \textbf{any semi-martingale whose jumps are totally inaccessible stopping times can be described via equation \eqref{eq:definitionX}}. With martingale random fields we can also consider some infinite dimensional cases, see \cite{DiNunno2010} for examples.

The goal is to find $\sup_u J(u)$ for controls adapted to the filtration $\FF$, where $X$ is adapted to the filtration $\GG$ and $\FF\subseteq \GG$, \ie for all $t\in [0,T]$ we have $\Ff_t \subseteq \Gg_t$. This is a problem with \emph{partial information} if $X$ is not $\FF$-adapted. 
We find (candidates for) optimal solutions by investigating
\begin{equation}
\frac{\partial}{\partial y} J(u+y\beta) \big|_{y=0}, \quad u,\; u+y\beta \text{ are admissible controls and } |y| <\delta,
\label{eq:intro_derivative}
\end{equation}
for some $\delta>0$. 
The controls are taking values in an open, convex set $\UU \subseteq \RR^n$. In the literature \eqref{eq:intro_derivative} has sometimes been evaluated using a set of assumptions that requires $\UU =\RR^n$. We explain this issue in Section \ref{A_remark_on_the_technique}, and state our maximum principle with weaker assumptions so that we can overcome this problem and indeed allow for $\UU \subsetneq \RR^n$.

Key to our approach is the non-anticipating derivative $\Dd$, an operator from  $L_2(\Omega, \Gg, \Prob)$ to the space of integrable random fields, see, e.g. \cite{DiNunno2002,DiNunno2010}. 
The operator $\Dd$ coincides with the dual of the It\^o non-anticipating stochastic integral with respect to a general martingale random field. Indeed we have that, for $\xi\in L_2(\Omega, \Gg, \Prob)$, 
\begin{align}
\E\Big[ \xi &\int\limits_0^T \int\limits_\Zz \kappa(s,z) \ins \mu(d s,d z) \Big] \nonumber = \E\Big[ \int\limits_0^T \int\limits_\Zz  (\Dd_{s,z} \xi) \kappa(s,z) \ins \Lambda(ds,dz) \big]. \label{eq:Duality_formula_intro}
\end{align}
Here $\Lambda$ represents the conditional variance measure associated to $\mu$. For continuous semi-martingales $\Lambda$ would be the quadratic variation, while for pure-jump semi-martingales $\Lambda$ would be the predictable compensator for the jumps (with respect to $\GG$). These concepts are further detailed in the forthcoming sections \ref{section:martinga_random_field_maximum}. 

\medskip

\begin{remark}
Here we will briefly discuss why this optimization problem cannot be (easily) solved by the usual BSDE-methods. In a BSDE-type approach, (see e.g. \cite{Peng1990, Framstad2004, Tang1994}) we would define a ``Hamiltonian''of type
\begin{equation*}
H(t,x,u,p,q,r) = f_t(u,x) + b_t(u,x) p_t + \int\limits_\Zz \phi_t(z) q_t(z) \ins \lambda_t(dz)
\end{equation*}
where $p$ and $q$ are solutions to the adjoint BSDE:
\begin{align}
dp_t &=  \frac{\partial}{\partial x} H(t,X_t,u_t,p_t, q_t) \ins dt + \int\limits_t^T \int\limits_\Zz q_t \ins \mu(dt,dz), \nonumber \\
p_T &= g'(X_T). \label{eq:BSDE_no_Solution}
\end{align}
The optimal solution is then described via conditions on $H$. Here $\lambda_t$ will be defined precisely in section 2, but if $\mu$ is a Brownian motion, then we just have $\Zz=\{0\}$ and $\lambda_t(dz) = 1$.

Does equation \ref{eq:BSDE_no_Solution} have a solution? Naturally, the answer depends on the noises in question and the requirements on $p$ and $q$. Suppose we require
\begin{enumerate}
\item  $\E[\sup_t p_t^2 ] <2$ 
\label{itemBSDE1}
\item $\E\big[\int_0^T q_t(z)^2 \Lambda(dt,dz)\big] < \infty$,
\label{itemBSDE2}
\item $p$ and $q$ are $\GG$-adapted.
\label{itemBSDE3}
\end{enumerate}
Suppose also that the martingale representation holds for $\GG$, \ie that every square integrable $\GG$-martingale $M$ has representation
\begin{equation*}
M_t = M_0 + \int\limits_0^t \int\limits_\Zz \eta_s(z) \ins \mu(ds,dz)
\end{equation*}
by means of a predictable, square integrable random field $\eta$. If the martingale represention property holds for $\GG$ in terms of $\mu$, then \eqref{eq:BSDE_no_Solution} will have a solution, at least for the mild conditions found in \cite{Jianming2000} (see also  \cite[Section 4.3]{Protter2005} on the topic of the martingale representation property). 
However, if the martingale representation does not hold for $\GG$, then \textbf{equation \eqref{eq:BSDE_no_Solution} may have no solution satisfying \ref{itemBSDE1}-\ref{itemBSDE2}-\ref{itemBSDE3}}. Indeed, the literature on optimization with BSDEs has focused on the cases where such a martingale representation is available.

One example where the martingale representation property does not hold is when $u$ has conditionally independent increments and $\GG$ is the filtration generated by the noise. In this case \eqref{eq:BSDE_no_Solution} may have no solutions \cite[Remark 4.6]{BSDE}. (In \cite{BSDE}, a solution can only be found by considering a filtration with anticipating information.) Other examples can also be found by, e.g., problems with partial information or letting $\GG$ be the filtration generated by $\mu$ and $g(X_T,\omega)$ involve a random variable that is not $\Gg_t$-measurable for $t<T$.
\end{remark}

Also note that when a direct BSDE-method is possible, our approach provides a new way of computing the adjoint equations.

\medskip




Maximum principles using the duality relation of the Malliavin derivative with the Skorohod integral have been studied in \cite{DiNunno2009b,Brandis2012}. This limits the study to L\'evy processes and, additionally, some restrictions are imposed to match the domains of the Malliavin derivative. Here we instead use the non-anticipating stochastic derivative, which enables us to treat very general martingale noises. Furthermore, in the case of L\'evy noise, we reduce assumptions of Malliavin differentiable random variables to square integrability. Since the non-anticipating derivative coincides with the Malliavin derivative when both are well defined, this extends previous results. Indeed, the proof of our maximum principle will borrow heavily from the ideas found in \cite{Brandis2012}.

For the portfolio problem with default risk, the main result is extended to a simpler sufficient condition for optimal control. Note that this example is not of L\'evy type, nor is the state process (in general) Markovian.


\medskip

In this paper, the maximum principle is studied in Section \ref{section:Maximum_principle}. But first we discuss the martingale random fields and stochastic non-anticipating derivative in Section \ref{section:martinga_random_field_maximum} and the details on the optimization problem in Section \ref{section:Optimization_problem}.  An important detail on the formulation of our type of maximum principle, that has previously been overlooked in the literature, is presented in Section \ref{A_remark_on_the_technique}.  
Section \ref{section:application_default_risk} presents an application to portfolio optimization in a market with assets subject to default risk.

\section{The martingale random field}
\label{section:martinga_random_field_maximum}

We now retrieve the stochastic integration and the non-anticipating stochastic derivative over a martingale random field $\mu$. We refer to \cite{DiNunno2010} for a detailed discussion on these concepts.

Let $(\Omega, \Gg,\Prob)$ be a complete probability space equipped with a right-continuous filtration $\GG := \{ \Gg_t, \, t\in [0,T]\}$. Let $\Zz$ be a separable topological space. We denote $\Bb_{\Zz}$ as the Borel $\sigma$-algebra on $\Zz$ and $\Bb_{[0,T]\times \Zz}$ as the Borel $\sigma$-algebra on the product space $[0,T]\times \Zz$. Note that $\Bb_{[0,T]\times \Zz}$ is generated by a semi-ring of sets of type
\begin{equation*}
\Delta = (t,s] \times Z, \quad 0 \leq t < s \leq T,\, Z \in \Bb_{\Zz}.
\end{equation*}

We say that the stochastic set function $\mu(\Delta)$, $\Delta \in \Bb_{[0,T]\times \Zz}$ is a martingale random field in $L_2(\Omega, \Gg, \Prob)$ on $[0,T]\times \Zz$ (with conditionally orthogonal values) with respect to $\GG$ if it satisfies the following properties \cite[Definition 2.1]{DiNunno2010}:
\begin{enumerate}
\item $\mu$ has a tight, $\sigma$-finite variance measure $V(\Delta) = E\big[\mu(\Delta)^2]$, $\Delta \in \Bb_{[0,T]\times \Zz}$, which satisfies $V(\{0\}\times \Zz)=0$. \label{enu:1}
\item $\mu$ is additive, \ie for pairwise disjoint sets $\Delta_1, \dots, \Delta_K$: $V(\Delta_k) < \infty$
\begin{equation*}
\mu \big( \bigcup_{k=1}^K \Delta_k) = \sum_{k=1}^K \mu(\Delta_k) 
\end{equation*}
and $\sigma$-additive in $L_2$. 
\label{enu:2}
\item $\mu$ is  $\GG$-adapted. \label{enu:3}
\item $\mu$ has the \emph{martingale property}. Consider $\Delta \subseteq (t,T] \times \Zz$. We have:
\begin{equation*}
\E\Big[ \mu(\Delta) \,\Big|\, \Gg_t \Big] = 0. 
\end{equation*}
\label{enu:4}
\item $\mu$ has conditionally orthogonal values. For any  $\Delta_1, \Delta_2 \subseteq (t,T]\times \Zz$ such that $\Delta_1 \cap \Delta_2 = \emptyset$ we have:
\begin{align*}
\E \Big[ \mu(\Delta_1) \mu (\Delta_2) \,\Big|\, \Gg_t \Big]  = 0.
\end{align*}
\label{enu:5}
\end{enumerate}

In particular, any finite sums of orthogonal, square integrable martingales would be a martingale random field in the sense of \ref{enu:1}-\ref{enu:2}-\ref{enu:3}-\ref{enu:4}-\ref{enu:5} above. In general, the filtration $\GG$ does not need to be the one generated by $\mu$.

The $\GG$-predictable $\sigma$-algebra on $\Omega\times [0,T]\times \Zz$ is denoted by $\Pp_{[0,T]\times \Zz}$ and is generated by sets of type 
\begin{equation*}
\Delta = A \times (t,s] \times Z, \quad 0 \leq t < s \leq T,\, Z \in \Bb_{\Zz},\, A\in \Gg_t.
\end{equation*}
The $\GG$-predictable $\sigma$-algebra $\Omega\times [0,T]$ is denoted by $\Pp_{[0,T]}$ and is generated by sets of type $\Delta = A \times (t,s], \;0 \leq t < s \leq T, A\in \Gg_t$. 
%
On $(\Omega\times [0,T]\times \Zz,\Pp_{[0,T]\times \Zz})$ the random field $\mu$ has a $\sigma$-finite conditional random variance measure \cite[Theorem 2.1]{DiNunno2010}. For $\GG$-martingale processes the conditional variance measure is the $\GG$-predictable compensator. 
We denote this conditional variance measure by $\Lambda$, and it has the following properties
\begin{align*}
\E\big[ \mu(\Delta)^2 \big|\Gg_t \big] &= \Lambda(\Delta), \quad \text{in $L_1(\Omega,\Gg,\Prob)$ for }\Delta \subseteq(t,T]\times \Zz, \\
\E \big[ \mu(\Delta)^2  \big] &= \E \big[ \Lambda(\Delta) \big].
\end{align*}
For later purposes we assume that $\Lambda$ is absolutely continuous with respect to the Lebesgue measure on $[0,T]$.
Namely we assume that there exists a transition kernel $\lambda$ from $(\Omega\times [0,T],\Pp_{[0,T]})$ to $(\Zz, \Bb_{\Zz})$ such that $\Lambda(\omega,dt,dz) = \lambda_t(\omega,dz)\ins dt$. Meaning that the mapping $(\omega,t) \to \lambda_t(\omega,Z)$ is $\Pp_{[0,T]}$ measurable for every $Z\in \Bb_{\Zz}$ and $\lambda_t(\omega, \cdot)$ is measure on $(\Zz, \Bb_{\Zz})$ for every $(\omega,t) \in \Omega\times [0,T]$. We refer to \cite{Cinlar2011} for further details on transition kernels.

We denote $\II$ as the set of $\GG$-predictable random fields $\phi :\Omega\times [0,T]\times \Zz\to\RR$ satisfying 
\begin{equation*}
\| \phi \|_{\II} := \E \Big[ \int\limits_0^T \int\limits_\Zz \phi(s,z)^2 \ins \lambda_s(dz) ds  \Big]^{\frac{1}{2}} < \infty.
\end{equation*}
We say that $\phi\in \II$ is a simple random field if it can be expressed as a finite sum of type
\begin{equation}
\phi(s,z,\omega) = \sum_{i=1}^N \phi_i(\omega) \ind_{\Delta_i}(s,z), 
\label{eq:simplefunction_inMax}
\end{equation}
where $\Delta_i = (t_i,s_i]\times Z_i$ and $\phi_i$ are bounded, $\Gg_{t_i}$-measurable random variables for $i=1,\dots N<\infty$. Simple, $\GG$-predictable random fields are dense in $\II$ by the usual It\^o integration type arguments and we have that, for every $\phi\in \II$:
\begin{align}
\E \Big[ \big( \int\limits_0^T \int\limits_\Zz \phi(s,z) \ins \mu(ds,dz) \big)^2 \Big] &= \E \Big[ \int\limits_0^T \int\limits_\Zz \phi(s,z)^2 \ins \Lambda(ds,dz) \Big] \nonumber \\
&= \E \Big[ \int\limits_0^T \int\limits_\Zz \phi(s,z)^2 \ins \lambda_s(dz) ds  \Big]. 
\label{eq:ito_isometry}
\end{align}
Remark also that $\phi \in \II$ implies that
\begin{equation*}
\int\limits_\Zz \phi_t(z)^2\ins \lambda_t(dz) < \infty, \quad dt\times d\Prob\text { a.e.} 
\end{equation*}
Note that $\int_0^t \int_\Zz \phi(s,z) \ins \mu(ds,dz)$, $t\in [0,T]$ is a $\GG$-martingale with values in $L_2$.


\medskip

The \emph{non-anticipating stochastic derivative} is a characterization of the integrand in the Kunita-Watanabe decomposition, developed in \cite{DiNunno2002,DiNunno2003,DiNunno2007,DiNunno2007b,DiNunno2010}. It is the adjoint (linear) operator $\Dd= I^*$ of the stochastic integral:
\begin{equation*}
\Dd: L_2(\Omega,\Gg,\Prob) \Longrightarrow \II.
\end{equation*}
A full characterization is given in constructive form using the elements of the following dissecting system, a sequence of partitions of $[0,T]\times \Zz$. 
Let $A_n \subseteq [0,T]\times \Zz$ be an increasing sequence of Borel-sets such that $V(A_n) < \infty$ for all $n\in{\mathbb{N}}$ and $\cup_{n=1}^\infty A_n =  [0,T]\times \Zz$. 
For every $n$ we chose a partition $\{\Delta_{n,k}\}$, $k=1,\dots, K_n < \infty$, of $A_n$ such that\footnote{Here it is possible to substitute $1/n$ with any sequence $\epsilon_n$ such that $\epsilon_n \to 0^+$ as $n\to \infty$.}
\begin{align}
& \bigcup_{1\leq k \leq K_n} \Delta_{n,k} = A_n, \label{eq:partition_1} \\
& \Delta_{n,k} = (t_{n,k},s_{n,k}]\times Z_{n,k},\quad 0\leq t_{n,k}< s_{n,k}  \leq T,\; Z_{n,k}\in \Bb_\Zz \label{eq:partition_2} \\
& \max_{1\leq k \leq K_n} (s_{n,k} - t_{n,k})  < 1/n,\label{eq:partition_4} \\
& \max_{1\leq k \leq K_n} V(\Delta_{n,k}) < 1/n, \label{eq:partition_3} \\
& \Delta_{n,k}\cap \Delta_{n,j} =\emptyset \text{ for } k\neq j. \label{eq:partition_5}
\end{align}
Moreover, the partitions are nested in the sense that
\begin{align}
& \Delta_{n,k}\cap \Delta_{n+1,j} =\emptyset \text{ or } \Delta_{n+1,j}. \label{eq:partition_6}
\end{align}
The non-anticipating stochastic derivative can be represented as the limit \cite[Theorem 3.1]{DiNunno2010}
\begin{equation}
\Dd \xi = \lim_{n\to\infty} \phi_n
\label{eq:DdG_limit_inMax}
\end{equation}
with convergence in $\II$ of the stochastic functions of type \eqref{eq:simplefunction_inMax} given by
\begin{equation}
\phi_n (t,z) := \sum_{k=1}^{K_n} \E \Big[ \xi \frac{\mu(\Delta_{n,k})}{\Lambda(\Delta_{n,k})} \Big| \Gg_{t_{n,k}} \Big] \ind_{\Delta_{n,k}}(t,z)
\label{eq:DdGlimit_simplefunctions_inMax}
\end{equation}
where $\Delta_{n,k}=(t_{n,k}, s_{n,k}]\times Z_{n,k}$ refers to the partion of $A_n$ described in \eqref{eq:partition_1}-\eqref{eq:partition_6}.
We have the following result \cite[Theorem 3.1]{DiNunno2010}:
\begin{teorem}
All $\xi \in L_2(\Omega, \Gg, \Prob)$ have representation
\begin{equation}
\xi =\xi_0+ \int\limits_0^T \int\limits_\Zz \Dd_{t,z} \xi  \ins\mu(dt,dz).
\label{eq:representation_with_Dd_inMax}
\end{equation}
Moreover $\Dd \xi_0 \equiv 0$ and $\xi_0 \in L_2(\Omega, \Gg, \Prob)$ is orthogonal to space generated by \\
$\Big\{\int_0^T\int_\Zz \phi(s,z)\ins \mu(ds,dz) \Big|\; \phi \in \II \Big\}$.
\label{theorem:representation_with_Dd_inMax}
\end{teorem}

Indeed, by the orthogonality of the sum in \eqref{eq:representation_with_Dd_inMax}, one can see that the following duality rule is verified: Let $\xi \in L_2(\Omega, \Gg, \Prob)$ and $\kappa \in \II$, then
\begin{align}
\E\Big[ \xi &\int\limits_0^T \int\limits_\Zz \kappa(s,z) ~\mu(d s,d z) \Big] \nonumber \\
&= \E\Big[ \Big(  \xi_0 + \int\limits_0^T \int\limits_\Zz \Dd_{s,z} \xi \ins\mu(d s,d z) \Big) \int\limits_0^T \int\limits_\Zz \kappa(s,z) \ins \mu(ds, dz) \Big] \nonumber \\
&= \E\Big[ \int\limits_0^T \int\limits_\Zz  (\Dd_{s,z} \xi) \kappa(s,z) \ins \Lambda(ds,dz) \big]. \label{eq:Duality_formula}
\end{align}

\section{Optimization problem}
\label{section:Optimization_problem}


Define the state process $X_t$, $t\in[0,T]$ by $X_0=a \in \RR$ and
\begin{equation*}
X_t^{(u)} = X_0 + \int\limits_0^t b_s(u_s,X_s) \ins ds + \int\limits_0^t \int\limits_\Zz \phi_s(z,u_s,X_{s\minus}) \ins \mu(ds,dz).
\end{equation*}
Here $b: \Omega \times [0,T]\times \UU \times \RR\to \RR$ and $\phi: \Omega \times [0,T]\times \Zz \times \UU \times \RR \to \RR$ are $\GG$-predictable. Moreover $\phi \in \II$. We assume that $X$ has an unique strong solution and note that $X$ is $\GG$-adapted. 
The stochastic process $u$ is the control taking values in an open and convex set $\UU \subseteq \RR^n$.

In the performance functional \eqref{eq:performance_functional_maximum_intro},
\begin{equation}
J(u) = \E \Big[ \int\limits_0^T f_t(u_t,X_t) \ins dt + g(X_T) \Big],
\label{eq:performance_functional_maximum}
\end{equation}
we have $f: \Omega\times [0,T]\times \UU \times \RR \to \RR$ and $g:\Omega \times \RR \to \RR$. Remark that we have allowed for $g$ and $f$ to depend on additional randomness besides $u$ and $X$, and assume that they are both measurable. 

We assume $f$ and $b$ are continuously differentiable in $x\in \RR$ and $u\in\UU$ for all $t\in [0,T]$ and almost all $\omega \in \Omega$. We denote these derivatives $\frac{\partial f_s}{\partial x}$, $\frac{\partial f_s}{\partial u}$, similarly for $b$ and $\phi$. Remark that $\frac{\partial f_s}{\partial u} \in \RR^n$ since $u$ is $n$-dimensional. We will denote $\cdot$ as the inner product in $\RR^n$ when appropiate. Furthermore $g$ is continuously differentiable with respect to $x\in\RR$ a.s., and we denote this derivative by $g'$.

The random field $\phi$ is continuously differentiable in $x\in \RR$ and $u\in \UU$ for almost all $(\omega,t,z)\in \Omega\times [0,T]\times \Zz$. We assume that $\frac{\partial \phi}{\partial x} \in \II$ and with $u =( u^{1}, \dots ,u^{n} ) \in \RR^n$, $\frac{\partial \phi}{\partial u^{j}} \in \II$ for $j=1,\dots n$. Finally we define the $\GG$-semi-martingale
%
\begin{equation*}
M_s := \int\limits_0^s \frac{\partial b_r}{\partial x}(u_r,X_r) \ins dr + \int\limits_0^s \int\limits_{\Zz}  \frac{\partial \phi_r}{\partial x}(u_r,X_r) \ins \mu(dr,dz),\quad  s\in [0,T]. \\
\end{equation*}
%

The \emph{first variation process} $G_s(t)$, $s\in[0,T]$, is the solution to the equation
\begin{align}
G_s(t) &:= 1 + \int\limits_t^s G_r(t) \ins d M_r, \quad s\in[t,T], \nonumber\\
&= 1 + \int\limits_t^s G_r(t)  \frac{\partial b_r}{\partial x}(u_r,X_r) \ins dr + \int\limits_t^s  \int\limits_{\Zz} G_r(t) \frac{\partial \phi_r}{\partial x}(z,u_r,X_r) \ins \mu(dr,dz). 
\label{eq:G_defined} 
\end{align}
The solution of \eqref{eq:G_defined} is given as follows (\cite[Theorem II.37]{Protter2005})
\begin{equation*}
G_s(t) = \exp \Big\{ M_s(t) -\frac{1}{2} [M(t) , M(t)]_s \Big\} \prod_{t<s\leq T} \big(1+\Delta M_s(t)\big ) \exp\{-\Delta M_s(t)\}
\end{equation*}
where for any $t$, $M(t)$ is the $\GG$-semi-martingale defined by $M_s(t) = \int_t^s \ins dM_r$ for $t<s\leq T$ and $M_s(t)=0$ for $s\leq t$.
Furthermore we define, where $t\in[0,T]$,
\begin{align}
K_t :=&\; K_t^{(u,X)} = g'(X_T)+\int\limits_t^T \frac{\partial f_s}{\partial x}(u_s,X_s) \ins ds \label{eq:K_defined}, 
\\
\Dd_{t,z} K_t :=&\; \Dd_{t,z} g'(X_T)+\Dd_{t,z} \big( \int\limits_t^T   \frac{\partial f_s}{\partial x}(u_s,X_s)  \ins ds\big), \label{eq:DK_defined} 
\displaybreak[0] \\
F_t(u,X_t) =&\; K_t \frac{\partial b_t}{\partial x} (u_t,X_t) +\int\limits_\Zz ( \Dd_{t,z} K_t ) \frac{\partial\phi_t}{\partial x}(z,u_t,X_t) \ins  \lambda_t(dz), 
\label{eq:F_defined} \\
p_t :=&\; p_t^{(u,X)}=  K_t + \int\limits_t^T F_s(u_s,X_s) G_s(t) \ins ds, 
\label{eq:p_defined} \displaybreak[0] \\
\kappa_t :=&\; \kappa_t^{(u,X)}= \Dd_{t,z} p_t. 
\label{eq:kappa_defined}
\end{align}
In order to have the above quantities well-defined the following requirements are needed:
\begin{antagelse} The control $u$ with state process $X^{(u)}$ satisfies
\begin{align}
\E \big[ g'(X_T)^2  \big] &< \infty, \label{eq:dg_integrable} \\
\E \big[ \int\limits_0^T \frac{\partial f_t}{\partial x} (u_t,X_t)^2 \ins dt \big] &< \infty, \label{dxf_integrable} \\
\E \big[ \int\limits_t^T \big( F_s G_s(t) \big)^2 \ins ds \big] &< \infty, \quad \text{for all } t\in [0,T]. \label{eq:FG_integrable}
\end{align}
\label{antagelse:integrerbarhet} 
\end{antagelse}

\begin{remark}
Suppose $\mu$ is a Brownian or L\'evy noise and $\GG$ is generated by $\mu$ with $\Gg = \Gg_T$. If using the duality relation of Malliavin calculus \eqref{eq:dg_integrable}-\eqref{dxf_integrable}-\eqref{eq:FG_integrable} would be stated in terms of Malliavin differentiability, see \cite[Equation 3.5]{Brandis2012}. Meaning that both $g'(X_T)$ and $\int_t^T \big( F_s G_s(t) \big)^2 \ins ds$ need to be in the domain of the Malliavin derivative, a space strictly smaller than $L_2(\Omega, \Gg, \Prob)$. In addition, \eqref{dxf_integrable} would be replaced by the Malliavin differentiability of $\frac{\partial f_t}{\partial x} (u_t,X_t)$ and the integrability of $\Md_t \frac{\partial f_t}{\partial x} (u_t,X_t)$ so that $\int_0^T \Md_t \frac{\partial f_t}{\partial x} (u_t,X_t) \ins dt$ would be well defined (where $\Md$ is the Malliavin derivative) since the arguments in the forthcoming \eqref{eq:interchange_stochastic_integration} does not apply.

\end{remark}

\medskip
For a given control $u$ with state process $X=X^{(u)}$, we define the Hamiltonian by
\begin{align}
\Ham_t(v,x) =&\; \Ham_t^{(u,X)}(v,x) \nonumber \\
:=&\; f_t(v,x) + b_t(v,x) p_t^{(u,X)} 
+ \int\limits_\Zz \kappa_t^{(u,X)} (z)\phi_t(z,v,x) \ins \lambda_t(dz),
\label{eq:Hamiltionian}
\end{align}
where $t\in [0,T]$, $v\in \UU$ and $x\in \RR$.

\section{Maximum principle}
\label{section:Maximum_principle}

Let $\FF := \{ \Ff_t, t\in[0,T]\}$ be a right continuous filtration such that $\Ff_t \subseteq \Gg_t$ for all $t\in[0,T]$. We state the optimization result for $\FF$, naturally we can have $\FF = \GG$.

\begin{definisjon}
We say that $u$ is an admissible control if $u:\Omega \times [0,T] \to \UU$ is $\FF$-predictable, 
Assumption \ref{antagelse:integrerbarhet} 
holds
and 
\begin{equation}
\E \Big[ \int\limits_0^T f_t(u_t,X_t)^2 \ins dt + g(X_T)^2 \Big] < \infty.
\label{eq:fg_integrable}
\end{equation}
We denote the set of admissible controls by $\Af$.
\end{definisjon}
The following assumption is needed for the controls on which we apply the maximum principle.
\begin{antagelse}
Let $u \in \Af$ be fixed. For this $u$ we assume that for any $\FF$-predictable and bounded process $\beta$ satisfying
\begin{equation}
u_t - \beta_t \in \UU \quad \text{and} \quad u_t+\beta_t \in \UU \quad dt\times d\Prob \text{ a.e.} 
\label{eq:inclusion_condition}
\end{equation}
%
there exists a $\delta>0$ such that
\begin{enumerate}
\renewcommand{\theenumi}{A\arabic{enumi})}
\item $u+y\beta\in \Af$ for all $|y| \leq \delta$.
\label{item:maximum_convexity}
\item 
The family 
\begin{equation}
\Big\{ \frac{\partial f_t}{\partial x}\big(u_t + y \beta_t,X^{u+y\beta}_t \big)\frac{\partial}{\partial y} X^{u+y\beta_t}_t + \frac{\partial f_t}{\partial u}\big(u_t + y \beta_t,X^{u+y\beta}\big)\beta_t \Big\}_{y\in(-\delta,\delta)}
\label{eq:maximum_uniform1}
\end{equation}
is uniformly $dt \times d\Prob$-integrable, and the family
\begin{equation}
\Big\{ g'\big(X_T^{u+y\beta}\big) \frac{\partial}{\partial y} X_T^{u+y\beta} \Big\}_{y\in(-\delta,\delta)}
\label{eq:maximum_uniform2}
\end{equation}
is uniformly $\Prob$-integrable. 
\label{item:uniform_integrability_maximum}
\item 
The process $Y_t^{(u,\beta)} = \frac{\partial}{\partial y} X^{u+y\beta}_t |_{y=0}$ exists as an element of $L_2(\Omega, \Gg, \Prob)$ for all $t\in [0,T]$ and satisfies
\begin{align}
Y_t =&\; Y_t^{(u,\beta)} = \frac{\partial}{\partial y} X^{u+y\beta}_t \big|_{y=0} \nonumber \\
=&\; \int\limits_0^t\big[ \frac{\partial b_s}{\partial x}(u_s,X_s)Y_s +\frac{\partial b_s}{\partial u}(u_s,X_s) \cdot\beta_s \big] \ins ds \nonumber \\
&+\int\limits_0^t\int\limits_\Zz \big[  \frac{\partial \phi_s}{\partial x}(z,u_s,X_s)Y_s +\frac{\partial \phi_s}{\partial u}(z,u_s,X_s)\cdot\beta_s \big] \ins \mu(ds,dz).
\label{eq:Y_sde}
 \end{align}
\label{item:Y_exist_maximum}
\end{enumerate}
\label{antagelse:integrasjon}
\end{antagelse}
In a converse conclusion in the forthcoming maximum principle, we will also require the following assumption:
\begin{antagelse} ~
\begin{enumerate}
\renewcommand{\theenumi}{A\arabic{enumi})}
\setcounter{enumi}{3}
\item If $\alpha$ is a random variable taking values in $\UU$ a.s. then (with $0\leq t<r\leq T$)
\begin{equation*}
u_s(\omega) = \alpha(\omega) \ind_{(t,r]}(s), 
\end{equation*}
is an admissible control (\ie $u\in \Af$).
\label{item:simple_admissible_maximum}
\end{enumerate}
\label{assumption:simple_admissible_maximum}
\end{antagelse}
%

A control $\hat{u}\in \Af$ is a ``local maximum'' if 
\begin{equation}
J(\hat{u}) \geq  J(\hat{u} + y\beta ), \quad |y| \leq \delta, 
\label{eq:local_maximum_maximum}
\end{equation}
for all bounded $\FF$-predictable $\beta$ and some $\delta>0$ that may depend on $\beta$. Meaning that we cannot improve $J(\hat{u})$ by making ``bounded'' pertubations of $\hat{u}$. 
Thus any solution to \eqref{eq:performance_functional_maximum}, $J(\hat{u})= \sup_{u\in \Af} J(u)$, must also be a local maximum. If $\hat{u}$ is a local maximum, we must naturally have
\begin{equation}
\frac{\partial}{\partial y} J(\hat{u}+y\beta) \big|_{y=0} =0. 
\label{eq:critical_point_intro_maximum}
\end{equation} 
The converse conclusions are not however true. Not every $u$ satisfying \eqref{eq:critical_point_intro_maximum} is a local maximum and a local maximum is not necessarily the optimal solution to \eqref{eq:performance_functional_maximum}.


\begin{teorem}
Let $\hat{u}$ be an admissible control and suppose $\hat{u}$ satisfies Assumption \ref{antagelse:integrasjon}. Denote 
\begin{align*}
 \hat{X}_t =&\; X_t^{(\hat{u})}\\
\hat{\Ham}_t(v,\hat{X}_t) =&\;  f_t(v,\hat{X}_t) + b_t(\lambda_t,v,\hat{X}_t) \hat{p}_t + \\
&+ \int\limits_\Zz \hat{\kappa}_t(z) \phi_t(z,v,\hat{X}_t) \ins \lambda_t(dz), \quad v\in \UU\subseteq \RR,
\end{align*}
with
\begin{align*}
\hat{p}_t &= p_t^{(\hat{u},\hat{X})}, \\
\hat{\kappa}_t &= \kappa_t^{(\hat{u},\hat{X})}.
\end{align*}
If $\hat{u}$ is a critical point for $J(u)$, in the sense that
\begin{equation*}
\frac{\partial}{\partial y} J(\hat{u}+y\beta) \big|_{y=0} =0 
\end{equation*}
for all bounded, $\FF$-predictable processes $\beta$ such that $\hat{u}_t \pm \beta_t \in \UU$ $dt\times d\Prob$-a.e., then
\begin{equation}
\E \Big[  \frac{\partial \Ham_t}{\partial u}(\hat{u}_t,\hat{X}_t) \Big| \Ff_t \Big] = 0, \quad dt\times d\Prob \text{-a.e.}
\label{eq:maximum_critical_point_condition}
\end{equation}
If Assumption \ref{assumption:simple_admissible_maximum} holds then the converse is also true: 
If $\hat{u}$ satisfies \eqref{eq:maximum_critical_point_condition} then $\hat{u}$ is a critical point.
\label{teorem:maximum_main_critical_point}
\end{teorem}
For ease of notation we use the short hand notation $b_s = b_s(\hat{u}_s,\hat{X_s})$, $f_s = f_s(\hat{u}_s,\hat{X_s})$, and similarly for the other coefficients.
%
\begin{proof}
Suppose $\hat{u}$ is a critical point.
Then
\begin{align}
0 &=  \frac{\partial}{\partial y} J(\hat{u}+y\beta)  \big|_{y=0} \nonumber \\
&= \E \big[ \int\limits_0^T \frac{\partial f_s}{\partial x} Y_s +\frac{\partial f_s}{\partial u} \cdot \beta_s \ins ds + g'(X_T) Y_T \Big] . \label{eq:maximum_critical_point}
\end{align}
By the duality formula \eqref{eq:Duality_formula} (and \eqref{eq:dg_integrable})
\begin{align}
\E& \big[ g'(X_T) Y_T \Big] \nonumber \\
 =&\; \E \Big[ \int\limits_0^T g'(X_T)\big[ \frac{\partial b_s}{\partial x} Y_s +\frac{\partial b_s}{\partial u} \cdot\beta_s \big] ds \nonumber \\
&+ \int\limits_0^T \int\limits_\Zz\big[  \big( \Dd_{s,z} g'(X_T) \big) \big( \frac{\partial \phi_s}{\partial x}(z) Y_s +\frac{\partial \phi_s}{\partial u}(z) \cdot\beta_s\big)  \ins \big] \Lambda(ds,dz) \Big].\label{eq:proof_duality}
\end{align}
By the Fubini theorem and the duality formula \eqref{eq:Duality_formula} (with integrability ensured by \eqref{eq:maximum_uniform1} and the non-anticipating stochastic derivative is well defined by \eqref{dxf_integrable})
\begin{align}
\E \big[& \int\limits_0^T \frac{\partial f_t}{\partial x} Y_t  \ins dt\big] \nonumber \\
=&\;\int\limits_0^T \E \Big[  \frac{\partial f_t}{\partial x} \big[ \int\limits_0^t\frac{\partial b_s}{\partial x} Y_s +\frac{\partial b_s}{\partial u} \cdot\beta_s  \ins ds \big]  \nonumber\\
&+\frac{\partial f_t}{\partial x}  \big[\int\limits_0^t\int\limits_{\Zz}  \big( \frac{\partial \phi_s}{\partial x} Y_s +\frac{\partial \phi_s}{\partial u} \cdot\beta_s \big) \ins \mu(ds,dz) \big] \Big]   \ins dt \nonumber \displaybreak[0] \\
=&\;\E \Big[ \int\limits_0^T \Big\{ \frac{\partial f_t}{\partial x} \big[ \int\limits_0^t\frac{\partial b_s}{\partial x} Y_s +\frac{\partial b_s}{\partial u} \cdot\beta_s  \ins ds \big]  \nonumber\\
&+\big[\int\limits_0^t\int\limits_{\Zz} \big( \Dd_{s,z}\frac{\partial f_t}{\partial x} \big) \big( \frac{\partial \phi_s}{\partial x} Y_s +\frac{\partial \phi_s}{\partial u} \cdot\beta_s \big) \ins \Lambda(ds,dz) \big] \Big\}\ins dt \Big] \nonumber \displaybreak[0]  \\
=&\;\E \Big[ \int\limits_0^T \big[ \int\limits_t^T \frac{\partial f_s}{\partial x} \ins ds  \big] \big( \frac{\partial b_t}{\partial x} Y_t +\frac{\partial b_t}{\partial u} \cdot\beta_t \big)  \ins dt \nonumber \\
&+\int\limits_0^T \int\limits_\Zz \big[ \int\limits_t^T  \Dd_{t,z} \frac{\partial f_s}{\partial x} \ins ds \big] \big( \frac{\partial \phi_t}{\partial x} Y_t +\frac{\partial \phi_t}{\partial u} \cdot\beta_t \big) \ins \Lambda(dt,dz) \Big]. \label{eq:proof_fubini}
\end{align}
By the continuity of $\Dd$ \cite[Remark 3.4]{DiNunno2010} and with sufficent integrability from \eqref{dxf_integrable} we have
\begin{equation}
\int\limits_t^T  \Dd_{t,z} \frac{\partial f_s}{\partial x} \ins ds = \Dd_{t,z} \int\limits_t^T \frac{\partial f_s}{\partial x} \ins ds, \quad d\Lambda\times d\Prob \text{ a.e.}
\label{eq:interchange_stochastic_integration}
\end{equation}
We recall \eqref{eq:K_defined}, \eqref{eq:DK_defined}, and by \eqref{eq:maximum_critical_point}-\eqref{eq:proof_duality}-\eqref{eq:proof_fubini} conclude that
\begin{align}
\E &\Big[ \int\limits_0^T K_s \big( \frac{\partial b_s}{\partial x} Y_s +\frac{\partial b_s}{\partial u} \cdot\beta_s \big) +\frac{\partial f_s}{\partial u} \cdot\beta_s \ins  ds \nonumber \\
&+ \int\limits_0^T \int\limits_\Zz (\Dd_{s,z}K_s) \big( \frac{\partial \phi_s}{\partial x} Y_s +\frac{\partial \phi_s}{\partial u} \cdot\beta_s \big) \ins\Lambda(ds,dz) \Big] = 0.
\label{eq:proof5}
\end{align}
Let $\alpha = (0,\dots, \alpha^{(j)}, \dots 0)$, be a random variable in $\RR^n$ which is zero except at the index $j$, where $1\leq j \leq n$. Set 
\begin{equation*}
 \beta_s= \alpha \ind_{(t,t+h]}(s) = (0,\dots, \alpha^{(j)}, \dots 0) \ind_{(t,t+h]}(s)
\end{equation*}
We assume $\alpha^{(j)}$ is bounded, $\Ff_t$-measurable and such that, $u_t\pm \beta_t$ takes values in $\UU$ $dt\times d\Prob$ a.e. Then $Y_s=Y^{(u,\beta)}_s= 0$ for $s<t$ so that \eqref{eq:proof5} can be rewritten as
\begin{equation}
A_1+A_2 = 0 
\label{eq:A1plusA2}
\end{equation}
where 
\begin{align*}
A_1 =&\; \E \Big[ \int\limits_t^T K_s\frac{\partial b_s}{\partial x} Y_s  \ins  ds+ \int\limits_t^T \int\limits_\Zz (\Dd_{s,z}K_s) \frac{\partial \phi_s}{\partial x} Y_s  \ins\Lambda(ds,dz) \Big], \\
A_2 =&\; \E\Big[\alpha\cdot \Big( \int\limits_t^{t+h} \big[ K_s\frac{\partial b_s}{\partial u} + \frac{\partial f_s}{\partial u} \big]\ins  ds+ \int\limits_t^{t+h} \int\limits_\Zz (\Dd_{s,z}K_s) \frac{\partial \phi_s}{\partial u} \ins \Lambda(ds,dz) \Big) \Big].
\end{align*}
From \eqref{eq:F_defined}
\begin{align*}
A_1 =&\;\E \Big[ \int\limits_t^T F_s Y_s \ins ds \Big] \\
=&\;\int\limits_{t}^{t+h} \E \Big[   F_s Y_s  \Big] ds +  \int\limits_{t+h}^{T} \E \Big[  F_s Y_s \Big]  \ins ds.
\end{align*}
Since $Y$ admits a c\`adl\`ag representative and $Y_t = 0$ we have
%
\begin{equation*}
\frac{\partial}{\partial h} \int\limits_{t}^{t+h} \E \Big[   F_s Y_s  \Big] ds \big|_{h=0} = 0.
\end{equation*}
Recall \eqref{eq:Y_sde} and \eqref{eq:G_defined}. We have 
\begin{equation*}
Y_s = Y_{t+h} G_s(t+h) \quad \text{for } s\geq t+h.
\end{equation*}
Since $Y_t = 0$ (interchange of integration and expectation justified by \eqref{eq:maximum_uniform1}, \eqref{eq:maximum_uniform2})
\begin{align*}
\frac{\partial}{\partial h} A_1 \big|_{h=0} &=  \frac{\partial}{\partial h} \int\limits_{t+h}^{T} \E \Big[  F_s Y_s \Big]  \ins ds \;\big|_{h=0} \displaybreak[0] \\
&= \int\limits_t^T  \frac{\partial}{\partial h} \Big\{ \E \Big[ F_s Y_{t+h} G_s(t+h) \Big]  \Big\} \ins ds \Big|_{h=0} -  F_t Y_t \displaybreak[0]  \\
&= \int\limits_t^T   \E \Big[ F_s \Big\{  Y_{t+h} \frac{\partial}{\partial h} G_s(t+h) + G_s(t+h) \frac{\partial}{\partial h}Y_{t+h}    \Big\}  \Big]  \ins ds \Big|_{h=0} \displaybreak[0] \\
&= \int\limits_t^T  \frac{\partial}{\partial h} \E \Big[  F_s  G_s(t) Y_{t+h}  \Big] \Big|_{h=0}  \ins ds.
\end{align*}
By \eqref{eq:Y_sde} we have
\begin{align*}
Y_{t+h} =&\; \alpha\cdot \Big( \int\limits_t^{t+h}\frac{\partial b_s}{\partial u} \ins ds +  \int\limits_t^{t+h} \int\limits_\Zz \frac{\partial \phi_s}{\partial u} \ins \mu(ds,dz) \Big) \\
&+ \int\limits_t^{t+h} Y_s\frac{\partial b_s}{\partial x} \ins ds +  \int\limits_t^{t+h} \int\limits_\Zz  Y_{s\minus} \frac{\partial \phi_s}{\partial x} \ins \mu(ds,dz) .
\end{align*}
Denote $\frac{\partial}{\partial h} A_1 |_{h=0} = B_1 + B_2$ with
\begin{align*}
B_1 &=  \int\limits_t^T  \frac{\partial}{\partial h} \E \Big[  F_s  G_s(t) \Big\{ \alpha\cdot \Big( \int\limits_t^{t+h} \frac{\partial b_r}{\partial u} \ins dr +  \int\limits_t^{t+h} \int\limits_\Zz \frac{\partial \phi_r}{\partial u} \ins \mu(dr,dz) \Big) \Big\}   \Big] \Big|_{h=0}  \ins ds , \displaybreak[0]  \\
B_2 &= \E \Big[ \int\limits_t^{T} \frac{\partial}{\partial h} \E \Big[  F_s  G_s(t) \Big\{  \int\limits_t^{t+h} Y_r \frac{\partial b_r}{\partial x} \ins dr +  \int\limits_t^{t+h} \int\limits_\Zz  Y_{r\minus} \frac{\partial \phi_r}{\partial x} \ins \mu(dr,dz) \Big\}   \Big] \Big|_{h=0}  \ins ds.
\end{align*}
By the duality formula \eqref{eq:Duality_formula} (well defined by \eqref{eq:FG_integrable})
\begin{align}
B_1 =&\; \int\limits_t^T  \frac{\partial}{\partial h} \E \Big[  F_s  G_s(t) \Big\{ \alpha\cdot \Big( \int\limits_t^{t+h} \frac{\partial b_r}{\partial u} \ins dr +  \int\limits_t^{t+h} \int\limits_\Zz \frac{\partial \phi_r}{\partial u} \ins \mu(dr,dz) \Big) \Big\}   \Big] \Big|_{h=0}  \ins ds \displaybreak[0] \nonumber \\
=&\; \int\limits_t^T  \frac{\partial}{\partial h} \E \Big[  \Big\{ \alpha\cdot \Big( \int\limits_t^{t+h}  F_s  G_s(t) \frac{\partial b_r}{\partial u}  \ins dr 
\nonumber \\
&+\int\limits_t^{t+h} \int\limits_\Zz \Dd_{r,z}\big( F_s  G_s(t) \big) \frac{\partial \phi_r}{\partial u} \ins \lambda_r(dz) \ins dr \Big)  \Big\}   \Big] \Big|_{h=0}  \ins ds \displaybreak[0]  \nonumber \\
=&\; \int\limits_t^T  \E \Big[  \Big\{ \alpha\cdot \Big( F_s  G_s(t) \frac{\partial b_t}{\partial u}  + \int\limits_\Zz \Dd_{t,z}\big( F_s  G_s(t) \big) \frac{\partial \phi_t}{\partial u} \ins \lambda_t(dz)  \Big) \Big\}   \Big]  \ins ds .\label{eq:B1_conclusion}
\end{align}

By the duality formula \eqref{eq:Duality_formula} (well defined by \eqref{eq:FG_integrable}) and since $Y_t = 0$ we have
\begin{align}
B_2 =&\;  \int\limits_t^{T} \frac{\partial}{\partial h}  \E \Big[ F_s  G_s(t)  \Big\{\int\limits_t^{t+h}  Y_r \frac{\partial b_r}{\partial x} \ins dr +  \int\limits_t^{t+h} \int\limits_\Zz  Y_{r\minus} \frac{\partial \phi_r}{\partial x} \ins \mu(dr,dz) \Big\}   \Big] \Big|_{h=0}  \ins ds \nonumber \displaybreak[0] \\
=&\;  \int\limits_t^T   \E \Big[\frac{\partial}{\partial h}   \Big\{ \int\limits_t^{t+h} F_s  G_s(t) Y_r \frac{\partial b_r}{\partial x} \ins dr \nonumber\\ 
&+  \int\limits_t^{t+h} \int\limits_\Zz  \Dd_{r,z}\big( F_s  G_s(t) \big)  Y_{r\minus} \frac{\partial \phi_r}{\partial x} \ins  \lambda_r(dz)\ins dr \Big\}   \Big] \Big|_{h=0}  \ins ds \nonumber \displaybreak[0] \\
=&\;  \int\limits_t^T  \E \Big[  \Big\{  F_s  G_s(t) Y_t \frac{\partial b_t}{\partial x}  +  \int\limits_\Zz  \Dd_{r,z}\big( F_s  G_s(t) \big)  Y_{t} \frac{\partial \phi_t}{\partial x} \ins  \lambda_t(dz)\Big\}   \Big]  \ins ds \nonumber \displaybreak[0] \\
=&\; 0. \label{eq:B2_conclusion}
\end{align}
\bigskip
We see immediately that (interchange of derivation and expectation justified by \eqref{eq:maximum_uniform1} \eqref{eq:maximum_uniform2})
\begin{align}
\frac{\partial}{\partial h} A_2 \Big|_{h=0} =&\; \E\Big[ \alpha\cdot \Big( K_t \frac{\partial b_t}{\partial u} + \frac{\partial f_t}{\partial u} + \int\limits_\Zz (\Dd_{t,z}K_t) \frac{\partial \phi_t}{\partial u} \ins \lambda_t(dz) \Big) \Big].
\label{eq:A2conclusion}
\end{align}
Recall that $\frac{\partial}{\partial h} A_1 = B_1+B_2$ and the definition of $p$ in \eqref{eq:p_defined}. By \eqref{eq:B1_conclusion}-\eqref{eq:B2_conclusion}-\eqref{eq:A2conclusion} we have
\begin{align}
\frac{\partial}{\partial h} \{ A_1+A_2 \}_{h=0} &= \E \Big[  \alpha\cdot \Big\{ \frac{\partial f_t}{\partial u} +  p_t \frac{\partial b_t}{\partial u} + \int\limits_\Zz \big( \Dd_{t,z} p_t \big) \frac{\partial \phi_t}{\partial u}(z) \ins \lambda_t(dz) \Big\} \Big] \nonumber \\
&= \E \Big[ \alpha\cdot \frac{\partial \Ham_t}{\partial u}(\hat{X}_t,\hat{u}_t) \Big]. \label{eq:alpha_in_U_conclucsion} 
\end{align}
As a function of $h$, $A_1(h)+A_2(h) = 0$ for all $0\leq h\leq T-t$ by \eqref{eq:A1plusA2}. Hence $\frac{\partial}{\partial h}\{ A_1(h)+A_2(h)\} = 0$ and thus 
\begin{equation*}
0=  \E \Big[ \alpha\cdot \frac{\partial \Ham_t}{\partial u}(\hat{X}_t,\hat{u}_t) \Big] = \E \Big[ \alpha^{(j)} \frac{\partial \Ham_t}{\partial u^{(j)} }(\hat{X}_t,\hat{u}_t) \Big].  
\end{equation*}

Recall that here $u_t+\alpha$ is a $\Ff_t$-measurable random variable taking values in $\UU$ a.s. Define
\begin{equation*}
D(\omega) = \sup_{c\in \RR} \{ u_t(\omega) + c \in \UU \text{ and }  u_t(\omega) - c \in \UU \}  \wedge 1. 
\end{equation*}
Here $D$ ``measures the minimum distance'' between $u^{(j)}$ and $\UU$ $\omega$-wise. Note that $0< D \leq 1$ a.s. Let $\zeta$ be a $\Ff_t$-measurable random variable bounded by $C>0$. Then 
\begin{equation*}
u_t + \frac{1}{2C} \zeta D\in \UU, \quad \text{a.s.}
\end{equation*}
We take $\alpha^{(j)} = \frac{1}{2C} \zeta \delta$ and from \eqref{eq:alpha_in_U_conclucsion} get
\begin{equation*}
\E\Big[ \frac{1}{2C} \zeta D \frac{\partial \Ham_t}{\partial u^{(j)}}(\hat{X}_t,\hat{u}_t) \Big] =0.
\end{equation*}
We multiply by $2C$ to find 
\begin{equation*}
\E\Big[ \zeta D  \frac{\partial \Ham_t}{\partial u^{(j)}}(\hat{X}_t,\hat{u}_t) \Big] =0.
\end{equation*}
Let $\zeta^{(m)} = \big( \frac{1}{\delta} \zeta  \big)  \wedge m$. Then $\zeta^{(m)} \to \zeta$ when $m\to\infty$ a.s. and we must have
\begin{equation*}
\E\Big[ \zeta \frac{\partial \Ham_t}{\partial u^{(j)}}(\hat{X}_t,\hat{u}_t) \Big] =0.
\end{equation*}
Since this holds for all $\Ff_t$ measurable $\zeta$ we conclude
\begin{equation}
\E\Big[  \frac{\partial \Ham_t}{\partial u^{(j)}}(\hat{X}_t,\hat{u}_t) \Big| \Ff_t \Big] = 0.
\label{eq:proof15}
\end{equation}
The proof for the sufficient condition is complete as \eqref{eq:proof15} holds for all $1\leq j \leq n$. 
%
%

\medskip
Conversely, suppose \eqref{eq:maximum_critical_point_condition}. By reversing the above argument we get that \eqref{eq:A1plusA2} holds for all $\beta \in \Af$ of the form 
\begin{equation*}
\beta(s,\omega) = \alpha(\omega) \ind_{(t,t+h]}(s), 
\end{equation*}
where the random variable $\alpha$ is $\Ff_t$-measurable, bounded and such that $u\pm\beta$, takes values in $\UU$ $dt\times d\Prob$ a.e. Here $0\leq t < t+h \leq T$. Hence \eqref{eq:A1plusA2} holds for all linear combinations of such $\beta$. Since any $\beta \in \Af$ can be approximated by such linear combinations it follows that \eqref{eq:A1plusA2} holds for all bounded $\beta \in \Af$.
\end{proof}

\section{A remark on the technique used}
\label{A_remark_on_the_technique}

In this paper, the maximum principle relies on on evaluating 
\begin{equation}
\frac{d}{dy} J(u+y\beta) 
\label{eq:technique}
\end{equation}
where $J$ is the performance functional \eqref{eq:performance_functional_maximum_intro}. Here $u$ is the control which is a ``candidate'' to be an optimal solution, $y\in \RR$ and $\beta$ is a pertubation of $u$. In this Section we discuss a technical point in how this technique has been presented in the literature, because some frequently used conditions have implications on how we can choose $\UU$ (the space where the controls are taking their values). In several papers, e.g. \cite{An2008a,An2010,baghery2007,Brandis2012,Haadem2013,Mataramvura2008,Meng2009,Pamen2009}, that evaluate \eqref{eq:technique} (for performance functionals of type \eqref{eq:performance_functional_maximum}, but with different assumptions on the noises) the following four assumptions are standard:
\begin{enumerate}
\item The admissible controls $u$ take values in an open, convex set $U \subseteq \RR^n$. 
\label{item:the_set}
\item The admissible controls satisfy some integrability conditions related to the problem and the corresponding state-process (given by a SDE) has a unique strong solution.
\label{item:other_conditions}
\item For all bounded and $\tilde{\Ff}_t$-measurable random variables $\alpha$, the control 
\begin{equation*}
u_s(\omega) = \alpha(\omega) \ind_{(t,t+r]}(s), \quad 0 \leq t < t+r \leq T,
\end{equation*}
is admissible\footnote{In \cite{baghery2007,Meng2009} it is only assumed that $\alpha$ takes values in $U$.}. Here $\tilde{\FF} = \{\tilde{\Ff}_t, \; t\in [0,T]\}$ is a filtration relevant to the optimization problem. 
\label{item:all_bounded}
\item If $u$ and $\beta$ are admissible controls, with $\beta$ bounded, there exist $\delta>0$ such that $u+y\beta$ is also an admissible control for all $|y| < \delta$.
\label{item:spaceconvexity}
\end{enumerate}
For convenience we only discuss the case when $n=1$ in Condition \ref{item:the_set}. However the issue presented here can easily be generalized to any $n>1$.

Condition \ref{item:all_bounded} implies that all the constants are elements of $U$, since $C \ind_{(t,t+r]}(s)$, $C\in \RR$ must be an admissible control. This can only be satisfied if $U=\RR$. Meaning that $U$ cannot be taken to be any open, convex set as described in \ref{item:the_set}, but it is necessary that $U = \RR$ for the maximum principle to apply.

We could attempt to change Condition \ref{item:all_bounded} to 
\begin{enumerate}
\renewcommand{\theenumi}{\roman{enumi}')}
\setcounter{enumi}{2}
\item For all bounded and $\tilde{\Ff}_t$-measurable random variables $\alpha$ such that $\alpha \in U$ a.s., the control 
\begin{equation*}
u_s(\omega) = \alpha(\omega) \ind_{(t,t+r]}(s)\quad 0 \leq t < t+r \leq T,
\end{equation*}
is admissible.
\label{item:alternate_bound}
\end{enumerate}
However, Condition \ref{item:spaceconvexity} is still a problem. To explain, suppose $U = (c_1,c_2)$ for some $c_1<c_2$ and let $\alpha$ be a bounded, $\tilde{\Ff}_t$-measurable random variable taking values in $U$. If condition \ref{item:alternate_bound} holds, both $u_s(\omega) := \alpha(\omega) \ind_{(t,t+r]}(s)$ and $\beta_s(\omega) := C \ind_{(t,t+r]}(s)$, $C\in (c_1,c_2)$, are admissible controls. Even if the random variable $\alpha$ satisfies $\alpha < c_2$ a.s. we can have
\begin{equation*}
\esssup \alpha = c_2
\end{equation*}
and thus $u_t+y \beta_t \in U$ a.s. is \emph{not} possible for any $y>0$. Hence $u_t+y \beta_t$ is not an admissible control for any $y>0$, as it is not taking values in $U$, and Condition \ref{item:spaceconvexity} fails.

The use of the ``standard'' assumptions \ref{item:the_set}-\ref{item:other_conditions}-\ref{item:all_bounded}-\ref{item:spaceconvexity} is not a major issue, the resulting maximum principle will hold for $U=\RR$. Indeed the technical conditions are correct even if opaque. Moreover, if one is only interested in bounded controls one can apply the maximum principle and then check whether the resulting control is in fact bounded. There will however be a problem, at least formally, if integrability conditions or other conditions (\ie \ref{item:other_conditions}) on the admissible controls require them to take values in a bounded set. Also, the study of the control problem with $U$ bounded has independent interest. As an example, in the forthcoming Proposition \ref{proposisjon:maximum_sufficiently_optimal} we show additional results on the uniquess of the solution when $U$ is bounded. Hence we used Assumption \ref{antagelse:integrasjon} in the maximum principle, Theorem \ref{teorem:maximum_main_critical_point}.

\section{Application to default risk}
\label{section:application_default_risk}

Here we show an application of the maximum principle to portfolio optimization. We choose a setting outside L\'evy processes that has independent interest: Assets with credit risk modeled by doubly stochastic Poisson processes. Credit risk with doubly stochastic Poisson processes has been widely studied in the literature, see e.g. \cite{Jarrow2001,Lando1998,Duffie2005}.

%

Let $\lambda_s=(\lambda_s^{(1)}, \dots \lambda_s^{(n)} )$, $s\in [0,T]$, be a positive, stochastic process in $\RR^n$. Let $\Lambda_t^{(j)} = \int_0^t \lambda_s^{(j)} \ins ds$, and denote the filtration generated by $\lambda$ as $\FF^\Lambda = \{ \Ff_t^{\Lambda},\; t\in [0,T]\}$. No assumptions of independence are required between $\Lambda^{(j)}$ and $\Lambda^{(k)}$ for any $j\neq k$.

The $n$-dimensional pure jump process $H_s= (H_s^{(1)}, \dots,H_s^{(n)})$ is a doubly stochastic Poisson process if, when conditioned on the $\lambda$'s, it is Poisson distributed. We assume that
\begin{equation*}
\Prob\big( H_t^{(j)} = k \ins\big| \Ff_T^\Lambda \big)  = \Prob\big( H_t^{(j)} = k  \ins\big| \Lambda^{(k)}_t \big)  = \frac{ (\Lambda_t^{(j)})^k }{k!} e^{-\Lambda_t^{(j)}} 
\end{equation*}
for all $1\leq j \leq n$ and $k\in \mathbb{N}$. Let $\Ht_t := H_t - \Lambda_t$, $t\in[0,T]$ and $\FF:= \{\Ff_t, t\in [0,T]\}$ be the filtration generated by $\Ht$. Let $\GG=\FF$ and $\Zz = \{1,\dots n\}$, where $\Zz$ is equipped with the discrete topology. Note that $\Ff_t^\Lambda \subset \Ff_t$ for all $t \in [0,T]$ by \cite[Theorem 2.8]{Chaos}. Then $\mu$ defined by $\mu(dt,z) = d\Ht_t^{(z)}$ is a martingale random field with respect to $\FF=\GG$ on $[0,T]\times \Zz$. 

Note that the non-anticipating stochastic derivative for doubly stochastic Poisson processes has been studied in \cite{Chaos}. Computational rules of Malliavin type can also be found in \cite{Yablonski2007}. 

Let $\tau^{(z)}$ be the first jump of $H^{(z)}$, $z=1,\dots n$.
We model each asset $S^{(z)}$ with a return $\rho^{(z)}+\lambda^{(z)}$ up to the time of default $\tau$. In the case of default the asset $S^{(z)}$ become worthless, \ie $S^{(z)}_{\tau^{(z)}} = 0$ (whenever $\tau^{(z)}<T$). The goal of the investor is to invest in the $n$ assets maximizing expected utility of the wealth at terminal time $T$. In mathematical terms: Let
\begin{align*}
S_t^{(1)} &= S_{t-}^{(1)} \ind_{\{\tau^{(1)}>t\}}(t) \big( \rho_t^{(1)} \ins dt - \ins d\Ht_t^{(1)} \big), \\
&\vdots \\
S_t^{(n)} &= S_{t-}^{(n)} \ind_{\{\tau^{(n)}>t\}}(t) \big( \rho_t^{(n)} \ins dt - \ins d\Ht_t^{(n)} \big) .
\end{align*}
Let $X$ denote the total wealth of the investor and the control $u$ denote the amount invested in the $n$ assets:
\begin{align*}
X_t =&\; \int\limits_0^t \sum_{z=1}^n\ind_{\{\tau^{(z)} >r\}}(r) u_r^{(z)} \rho_r^{(z)} \ins dr -  \int\limits_0^t \sum_{z=1}^n \ind_{\{\tau^{(z)}>r\}}(r) u_r^{(z)}  \ins  d \Ht_t^{(j)} \\    
\end{align*}
Remark that every asset $S^{(z)}$ and the wealth process $X$ are $\FF$-adapted. 
With 
\begin{equation}
J(u) = \E \Big[ U(X_T) \Big] 
\label{eq:default_optimization}
\end{equation}
where $U: \RR\to \RR$ is an utility function (differentiable, increasing and strictly concave), we look for
\begin{equation*}
\sup_{u\in \Af} J(u).
\label{eq:default_optimization_goal}
\end{equation*}
We have
\begin{align*}
K_t &= U'(X_T), \\ 
F_t &= U'(X_T) \sum_{z=1}^n\ind_{\{\tau^{(z)} >t\}}(t) u_t^{(z)} \rho_t^{(z)} + \sum_{z=1}^n \big( \Dd_{t,z} U'(X_t) \big) \ind_{\{\tau^{(z)}>t\}}(t) u_t^{(z)} \lambda_t^{(z)}, \\
p_t &= U'(X_T), \\
\kappa_t &= \sum_{z=1}^n \Dd_{t,z} U'(X_t) ,\\
G_s(t&) = 0.
\end{align*}
Remark that under these assumptions, any $\FF$-predictable process $u$ is an admissible control if
\begin{equation}
 \E\big[ U(X^u_T)^2 + U'(X^u_T)^2 \big] <\infty.
\label{eq:easy_admissible}
\end{equation}
Furthermore Assumption \ref{antagelse:integrasjon} only depends on verifying \eqref{eq:easy_admissible} for $u+y\beta$. The Hamiltonian \eqref{eq:Hamiltionian} is given by
\begin{align*}
\Ham_t(u,x) =&\; U'(X_T) \sum_{z=1}^n\ind_{\{\tau^{(z)} >t\}}(t)  u_t^{(z)} \rho_t^{(z)} \\
&+ \sum_{z=1}^n \big( \Dd_{t,z} U'(X_t) \big) \ind_{\{\tau^{(z)}>t\}}(t) u_t^{(z)} \lambda_t^{(z)}. 
\end{align*}
Hence
\begin{equation*}
\frac{\partial \Ham_t}{\partial u}(v,x) = U'(X_T) \sum_{z=1}^n\ind_{\{\tau^{(z)} >t\}}(t) \rho_t^{(z)} + \sum_{z=1}^n \big( \Dd_{t,z} U'(X_T) \big) \ind_{\{\tau^{(z)}>t\}}(t)  \lambda_t^{(z)}.
\end{equation*}

Theorem \ref{teorem:maximum_main_critical_point} finds critical points for \eqref{eq:default_optimization}. To ensure that a critical point $\hat{u}$ is a solution to \eqref{eq:default_optimization_goal} we need to know that 1) the critical point is a local maximum and 2) there are no other critical points $\bar{u}$ where $J(\bar{u})>J(\hat{u})$. We investigate the exact properties of the critical points in Proposition \ref{proposisjon:maximum_sufficiently_optimal} and sufficent conditions for a solution to \eqref{eq:default_optimization_goal} are given in Corollary \ref{korollar:supremum_exists}.

\begin{proposisjon}
Assume that
\begin{enumerate}
\item $U$ is twice continuously differentiable and concave,
\item All bounded $\FF$-predictable processes taking values in $\UU$ are admissible controls,
\label{item:maximum_convexity_bounded}
\item For any $u \in \Af$ and $\FF$-predictable bounded process $\beta$ such that 
\begin{equation}
u_t \pm \beta_t \in \UU, 
\quad dt\times d\Prob \text{ a.e.} 
\label{eq:plusminusinU}
\end{equation}
then there exist $\varepsilon>0$ such that 
\begin{align}
\Big \{ U''\big(X^{u+y\beta}_T \big) \Big( &\int\limits_0^t \sum_{z=1}^n\ind_{\{\tau^{(z)} >r\}}(r) \beta^{(z)} \rho_r^{(z)} \ins dr \nonumber \\
&- \sum_{z=1}^n \int\limits_0^t  \ind_{\{\tau^{(z)}>r\}}(r) \beta^{(z)}  \ins  d \Ht_t^{(z)} \Big)^2 \Big\}_{y\in
(-\varepsilon,\varepsilon)} \label{eq:maximum_second_uniform}
\end{align}
is uniformly $\Prob$-integrable,
\item Assumption \ref{antagelse:integrasjon} holds for all bounded $u \in \Af$.
\end{enumerate}
Let $\epsilon = min(\delta, \varepsilon)$, where $\delta$ is as in \eqref{eq:maximum_uniform2}. 
Then the mapping $y \to J(u+y\beta)$, $y\in(-\epsilon,\epsilon)$, is strictly concave for all $u\in\Af$ and bounded $\FF$-predictable $\beta$ satisfying \eqref{eq:plusminusinU}. Furthermore, there is at most one bounded $u \in \Af$ such that $u$ is a critical point (in the sense of Theorem \ref{teorem:maximum_main_critical_point}).
\label{proposisjon:maximum_sufficiently_optimal}
\end{proposisjon}
%
\begin{proof}
First we prove the concavity of the mapping $y \to J(u+y\beta)$, $y\in (-\epsilon,\epsilon)$. We interchange the derivation and expectation and get
\begin{align*}
\frac{\partial^2}{\partial y^2} J(u+y\beta) =&\;  \E \Big[ \frac{\partial^2}{\partial y^2} U(X^{u+y\beta}_T) \Big] \\
=&\; \E \Big[  U''\big(X^{u+y\beta_T}\big) \Big( \int\limits_0^t \sum_{z=1}^n\ind_{\{\tau^{(z)} >r\}}(r) \beta^{(z)} \rho_r^{(z)} \ins dr \\
&- \sum_{z=1}^n \int\limits_0^t  \ind_{\{\tau^{(z)}>r\}}(r) \beta^{(z)}  \ins  d \Ht_t^{(z)} \Big)^2 \Big] < 0,
\end{align*}
where the last inequality follows by the concavity of $U$. 
\smallskip \noindent

Next we want to show that there is at most one bounded $u \in \Af$ such that $u$ is a critical point. 
First we show that when $u \in \Af$ is bounded and $\beta$ is as in \eqref{eq:plusminusinU}, we have $\epsilon>1$, \ie that $y \to J(u+y\beta)$ is a strictly concave mapping for $y\in(-\epsilon,\epsilon)$ with $\epsilon>1$. 
The claim $\epsilon>1$ follows from \ref{item:maximum_convexity_bounded} and the uniform integrability conditions \eqref{eq:maximum_uniform1}-\eqref{eq:maximum_uniform2}-\eqref{eq:maximum_second_uniform} since $\frac{\partial}{\partial y} J(u+y\beta)|_{y=a} = \frac{\partial}{\partial y} J(u+a\beta+y\beta) |_{y=0}$. 

\smallskip\noindent
Suppose $\bar{u},\hat{u} \in \Af$ are both bounded and critical points.
Set $\beta_t = \bar{u}_t-\hat{u}_t$.
Consider the control $\hat{u}+\frac{1}{2}\beta \in \Af$, and the mapping 
\begin{equation}
h(y) \to J\big(\hat{u}+\frac{1}{2}\beta+y\frac{1}{2}\beta\big), \quad y\in (-\epsilon,\epsilon)
\label{eq:concave_mapping_maximum}
\end{equation}
Note that $\epsilon>1$, $h(1) = J(\bar{u})$ and $h(-1) = J(\hat{u})$. Since $h$ is strictly concave at most one of $h(-1)$ and $h(1)$ can be a maximum.

%
\end{proof}
%

\begin{korollar}
Suppose the Assumptions in Proposition \ref{proposisjon:maximum_sufficiently_optimal} hold. If $\UU$ is bounded and a critical point $\hat{u}$ exists, then $\hat{u}$ is optimal, \ie
\begin{equation*}
J(\hat{u}) =  \sup_{u\in \Af} J(u),
\end{equation*}
and optimal portfolio $\hat{u}$ is characterized by
\begin{align*}
\E \big[ \frac{\partial \Ham_t}{\partial u}(\hat{u},X_t^{(\hat{u})}) \big| \Ff_t \big] =&\; \sum_{z=1}^n\ind_{\{\tau^{(z)} >t\}}(t) \rho_t^{(z)} \E\big[ U'(X_T) \ins \big| \Ff_t \big]  \\
&+ \sum_{z=1}^n \big( \Dd_{t,z} U'(X_T) \big) \ind_{\{\tau^{(z)}>t\}}(t)  \lambda_t^{(z)} = 0,
\end{align*}
for all $t\in [0,T]$ a.s.
\label{korollar:supremum_exists}
\end{korollar}

\begin{proof}
This is a restatement of Proposition \ref{proposisjon:maximum_sufficiently_optimal}.
\end{proof}

\section{Acknowlegdements}
I would like to thank Giulia Di Nunno and Bernt {\O}ksendal for valuable comments and discussions during the development of this paper.

The research leading to these results has received funding from the European Research Council under the European Community's Seventh Framework Programme (FP7/2007-2013) / ERC grant agreement no [228087].
\bibliographystyle{alpha}
\phantomsection
\addcontentsline{toc}{section}{References}
\bibliography{referanser} 
 
\end{document}